\documentclass[12pt]{amsart}

\usepackage{setspace}
\usepackage{amsmath, amssymb,amsthm, amsfonts}
\newtheorem{thm}{Theorem}
\newtheorem{lma}{Lemma}
\newtheorem{prop}{Proposition}
\newtheorem{cor}{Corollary}

\newtheorem*{thm*}{Theorem}
\theoremstyle{definition}

\theoremstyle{remark}
\newtheorem{remark}{Remark}

\newtheorem{example}{Example}

\newcommand{\Ric}{\mbox{Ric}}
\newcommand{\R}{\mathbb R}

\newcommand{\be}{\begin{equation}}
\newcommand{\ee}{\end{equation}}
\newcommand{\bee}{\begin{equation*}}
\newcommand{\eee}{\end{equation*}}
\def\D{\Delta_f}

\def\na{\nabla}

\def\la{\langle}
\def\ra{\rangle}

\def\Pi{\displaystyle{\mathbb{II}}}

\def\a{\alpha}

\begin{document}

\title[]{Curvature estimates for steady Ricci solitons}

\author{Pak-Yeung Chan}
\address{School of Mathematics,
University of Minnesota, Minneapolis, MN 55455, USA} \email{chanx305@umn.edu}

\maketitle
\markboth{Pak-Yeung Chan} {Curvature estimates for steady Ricci solitons}

\begin{abstract} We show that for an $n$ dimensional complete non Ricci flat gradient steady Ricci soliton with potential function $f$ bounded above by a constant and curvature tensor $Rm$ satisfying $\overline{\lim}_{r\to \infty} r|Rm|<\frac{1}{5}$, then $|Rm|\leq Ce^{-r}$ for some constant $C>0$, improving a result of \cite{MunteanuSungWang-2017}. For any four dimensional complete non Ricci flat gradient steady Ricci soliton with scalar curvature $S\to 0$ as $r\to \infty$, we prove that $|Rm|\leq cS$ for some constant $c>0$, improving an estimate in \cite{CaoCui-2014}.  As an application, we show that for a four dimensional complete non Ricci flat gradient steady Ricci soliton, $|Rm|$ decays exponentially provided that $\overline{\lim}_{r\to \infty} rS$ is sufficiently small and $f$ is bounded above by a constant.
\end{abstract}

\section{Introduction}The notion of Ricci solitons was first introduced by Hamilton in \cite{Hamilton-1988}. Let $(M^n,g)$ be an $n$ dimensional Riemannian manifold and $X$ be a smooth vector field on $M$, the triple $(M,g,X)$ is said to be a Ricci soliton if there is a constant $\lambda$ such that the following equation is satisfied
\be\label{eq-RS-1}
\Ric+\dfrac{1}{2}L_{X}g=\lambda g,
\ee
where $Ric$ and $L_{X}$ denote the Ricci curvature and Lie derivative with respect to $X$ respectively. A Ricci soliton is called shrinking (steady, expanding) if $\lambda>0 (=0, <0)$. It is said to be gradient if $X$ can be chosen such that $X=\nabla f$ for some smooth function $f$ on $M$.  The soliton is called complete if $(M,g)$ is complete as a Riemannian manifold.

Ricci flow was introduced by Hamilton in his seminal work \cite{Hamilton-1982} and deforms the metric in the direction of its Ricci curvature:
\be\label{eqn of RF}
\frac{\partial g(t)}{\partial t}=-2\Ric(g(t)).
\ee
Since then, it has been one of the most important flows in geometric analysis. Due to its nonlinear nature, the flow may develop different types of singularities. A gradient Ricci soliton, as a self similar solution to the Ricci flow, often arises as a rescaled limit of the flow near its singularities. The study of Ricci soliton is therefore crucial for understanding the singularity formation of the Ricci flow \cite{Hamilton-1995}. Indeed, it played a significant role in the resolution of Poincare's conjecture by Perelman \cite{Perelman-2002}, \cite{Perelman-2003} and \cite{Perelman-2003.2}. Please consult the excellent survey paper \cite{Cao-2010} by Cao for more details in this direction.

Ricci soliton can also be viewed as a generalization of the Einstein metric $\Ric=\lambda g$. Bakry and Emery first introduced the Bakry Emery Ricci curvature $\Ric_f:=\Ric+\nabla^2 f$ in \cite{BakryEmery-1985}. The Bakry Emery curvature is closely related to the diffusion processes, logarithmic Sobolev inequalities and isoperimetric inequalities (see \cite{Lott-2003}, \cite{CarrilloNi-2009} and \cite{WeiWylie-2009}). Since $L_{\nabla f}g=2\nabla^2 f$, the gradient Ricci solitons equation \eqref{eq-RS-1} can be rewritten as
\be\label{eq-RS-2}
\Ric_{f}=\Ric+\nabla^2 f=\lambda g,
\ee
which is a natural extension of the Einstein metric.

%The curvature estimate of non-compact Ricci soliton is of great interest since the behavior of the curvature at infinity controls the geometry of the manifolds at infinity. For instances, Munteanu-Wang \cite{MunteanuWang-2015.3} showed that any $4$ dimensional gradient shrinker is smoothly asymptotically to cone if its scalar curvature $S$ $\to$ $0$ as $r\to \infty$. Deng-Zhu \cite{DengZhu-2016} improved a result of Brendle \cite{Brendle-2013} and showed that any 3 dimensional complete gradient steadier with linear scalar curvature decay must be rotationally symmetric. Chen-Deruelle \cite{ChenDeruelle-2015} (see also \cite{Deruelle-2017}) proved the existence of conical structure for gradient expanding Ricci soliton under $\overline{\lim}_{r\to\infty} r^2|Rm|<\infty$. Moreover, via curvature estimate, we can know more about the compactness of different classes of Ricci soliton. Munteanu-M. T. Wang \cite{MunteanuWangMT-2011} showed that a sequence of gradient Ricci shrinkers with uniform bounded $\Ric$ and certain local integral bound of $Rm$ will converge subsequentially in pointed Cheeger Gromov sense to an orbifold gradient shrinker.

Recently, there has been lots of researches concerning the curvature estimate of gradient Ricci soliton, for example, \cite{ChowLuYang-2011}, \cite{MunteanuWangMT-2011}, \cite{MunteanuWang-2015.3}, \cite{CaoCui-2014}, \cite{DengZhu-2015}, \cite{MunteanuWang-2016}, \cite{MunteanuSungWang-2017}, \cite{Deruelle-2017}, \cite{LiNiWang-2018}, \cite{Chen-2018}, \cite{LiNi-2019} and \cite{ChowFreedmanShinZhang-2019}. Chow-Freedman-Shin-Zhang \cite{ChowFreedmanShinZhang-2019} showed that any four dimensional complete gradient steady Ricci soliton arising as a blow up limit of finite time singularity of Ricci flow on closed manifold must have bounded curvature. For any four dimensional complete gradient shrinking Ricci soliton arising from similar way, they proved that the Riemann curvature $Rm$ of such shrinking Ricci soliton has at most quadratic growth. In \cite{MunteanuSungWang-2017}, Munteanu-Sung-Wang studied the solvability of weighted Poisson equation for certain class of smooth metric measure space. As an application, they proved the following:

\begin{thm}\cite{MunteanuSungWang-2017}\label{MSW curvature estimate} Let $(M^n, g,f)$ be an $n$ dimensional complete non Ricci flat gradient steady Ricci soliton with $n\geq 2$. Suppose the potential function $f$ is bounded above by a constant and  $\displaystyle\lim_{r\to \infty}r|Rm|=0$, then there exists a positive constant $C$ such that
\be\label{estimate of MSW}
|Rm|(x)\leq C(1+r(x))^{3(n+1)}e^{-r(x)},
\ee
where $r=r(x)$ is the distance of $x$ from a fixed point $p_0$ $\in M.$
\end{thm}
It is not known whether the decay rate on the right hand side of (\ref{estimate of MSW}) is sharp or not. In this paper, we shall sharpen the upper bound under a weaker curvature decay condition. Instead of applying the Green's function estimate in \cite{MunteanuSungWang-2017}, we adopt the method based on the maximum principle introduced by Deruelle \cite{Deruelle-2017} in order to study the curvature properties of expanding gradient Ricci solitons. Here is the main result of the paper:
\begin{thm} \label{decay estimate in dim n} Let $(M^n,g,f)$ be an $n$ dimensional complete non Ricci flat gradient steady Ricci soliton with $n\geq 2$. Suppose the potential function $f$ is bounded above by a constant and the curvature tensor $Rm$ satisfies $\displaystyle\limsup_{r\to \infty}   r|Rm|<\frac{1}{5}$. Then there exists a positive constant $C$ such that
\be\label{decay estimate in dim n ineq}
|Rm|(x)\leq Ce^{-r(x)} \text{  on  } M,
\ee
where $r=r(x)$ is the distance of $x$ from a fixed point $p_0$ $\in M.$
\end{thm}
\begin{remark}The decay rate is sharp on $\Sigma\times \mathbb{T}^{n-2}$, where $\Sigma$ and $\mathbb{T}^{n-2}$ denote the Hamilton's cigar soliton and $n-2$ dimensional flat torus respectively (see Example \ref{cigar soliton} in Section 2).
 \end{remark}

%\begin{remark} For the $n$ dimensional Bryant soliton, $\lim_{r\to \infty}r|Rm|=\sqrt{\frac{(n-1)}{2(n-2)}}>\frac{1}{5}$ (see \cite{Brendle-2013} and \cite{Brendle-2014}). One expects the constant $\frac{1}{5}$ in Theorem \ref{decay estimate in dim n} is not optimal.
%\end{remark}

%\begin{remark} The constant $\frac{1}{5}$ in Theorem \ref{decay estimate in dim n} is technical and arises from a differential inequality satisfied by $|Rm|$. For the $n$ dimensional Bryant soliton, $\lim_{r\to \infty}r|Rm|=\sqrt{\frac{(n-1)}{2(n-2)}}>\frac{1}{5}$ (see \cite{Brendle-2013} and \cite{Brendle-2014}). One expects the constant $\frac{1}{5}$ is not optimal.
%\end{remark}

\begin{remark} The constant $\frac{1}{5}$ in Theorem \ref{decay estimate in dim n} is technical. For the $n$ dimensional Bryant soliton, $\lim_{r\to \infty}r|Rm|=\sqrt{\frac{(n-1)}{2(n-2)}}>\frac{1}{5}$ (see \cite{Brendle-2013} and \cite{Brendle-2014}). It is unclear at this point whether the constant $\frac{1}{5}$ is optimal.
\end{remark}

 \begin{remark}
 In view of Lemma \ref{eqv cond for proper of f}, the conclusion (\ref{decay estimate in dim n ineq}) is still true under a seemingly weaker condition on $f$, namely, $\overline{\lim}_{r\to\infty} r^{-1}f<1$, for the sake of simplicity, we state the theorem in the above way. $\Ric \geq 0$ and $S\to 0$ at infinity are sufficient for $f$ to be bounded above by a constant (see \cite{CarrilloNi-2009}).
  \end{remark}
\begin{remark}
  The quantity $r|Rm|$ is not scaling invariant, we are using the scaling convention such that (\ref{eqn of naf}) holds (see Section 2).
\end{remark}
%In view of theorem \ref{MSW curvature estimate} and \ref{decay estimate in dim n}, one possible conjecture concerning the decay of curvature of gradient steady Ricci soliton is that, the decay rate is either linear or exponential. It would be very useful for the classification of steady soliton If one can establish this dichotomy. For exponential decay of curvature, Deng-Zhu \cite{DengZhu-2018} generalized a result of \cite{MunteanuSungWang-2017} and showed that any non-flat complete gradient steadier with non-negative sectional curvature is a quotient of $\R^{n-2}\times \Sigma$ if $rS$ is small near infinity, where $\Sigma$ is the cigar soliton. Deng-Zhu \cite{DengZhu-2016} improved a result of Brendle \cite{Brendle-2013} and proved that any non-flat three dimensional gradient steady soliton with scalar curvature decays exactly linearly is the Bryant soliton. See \cite{Brendle-2014} and \cite{DengZhu-2017} for more classification results of higher dimensional steadier with curvature decay conditions.

In view of Theorem \ref{MSW curvature estimate} and \ref{decay estimate in dim n}, one possible conjecture concerning the curvature
decay of gradient steady Ricci soliton is that, the decay rate is either linear or exponential. It would be very useful for the classification of steady soliton if one can establish this dichotomy. There are a number of classification results under certain curvature decay conditions.  For exponential curvature decay, Deng-Zhu \cite{DengZhu-2018} generalized a result of \cite{MunteanuSungWang-2017} and showed that any non-flat complete gradient steady Ricci soliton with non-negative sectional curvature is a quotient of $\R^{n-2}\times \Sigma$ if $rS$ is small near infinity, where $\Sigma$ and $S$ denote the cigar soliton and scalar curvature respectively (see also \cite{Deruelle-2012} and \cite{CatinoMastroliaMonticelli-2016}). For linear curvature decay, Brendle \cite{Brendle-2013} solved a conjecture of Perelman \cite{Perelman-2002} and proved that any three dimensional non-flat and non-collapsed gradient steady Ricci soliton is the Bryant soliton.  Deng-Zhu \cite{DengZhu-2016}, \cite{DengZhu-2018} improved the result of Brendle \cite{Brendle-2013} and proved that any non-flat three dimensional gradient steady soliton with linear curvature decay is the Bryant soliton. See \cite{Brendle-2014} and \cite{DengZhu-2017} for more classification results of higher dimensional steady Ricci soliton with curvature decay conditions.

It is also worth mentioning that very lately Brendle \cite{Brendle-2018} confirmed a conjecture by Perelman \cite{Perelman-2002} and showed that any three dimensional complete non-compact, non-collapsed ancient solution to the Ricci flow with bounded positive sectional curvature is eternal and isometric to the Bryant soliton (see also \cite{LiZhang-2018}, \cite{BamlerKleiner-2019} and \cite{Brendle-2019}).

The second part of the paper is devoted to the estimate in dimension four. Munteanu-Wang \cite{MunteanuWang-2015.3} proved that a complete four dimensional gradient shrinking Ricci soliton must have bounded Riemann curvature $Rm$ if its scalar curvature $S$ is bounded (see \cite{CaoCui-2014} for the estimate of steady soliton by Cao-Cui). Furthermore, $|Rm|$ is controlled by $S$ pointwise, more precisely, there exists a constant $c>0$ such that
\be\label{MW 2015 estimate}
|Rm|\leq cS \text{  on  } M.
\ee
In dimension three, (\ref{MW 2015 estimate}) is a direct consequence of the Hamilton-Ivey estimate which implies that three dimensional complete gradient shrinking and steady Ricci solitons have non-negative sectional curvature (\cite{Ivey-1993}, \cite{Hamilton-1995} and \cite{Chen-2009}). However in dimension four, the sectional curvature of shrinking and steady Ricci solitons may change sign (see \cite{Cao-2010}).
%The estimate (\ref{MW 2015 estimate}) can be viewed as a four dimensional analog of Hamilton-Ivey pinching estimate for soliton.

Recently, Conlon-Deruelle-Sun \cite{ConlonDeruelleSun-2019} classified all non-flat gradient K$\ddot{a}$hler shrinking Ricci soliton of complex dimension two with $\lim_{r\to\infty}S=0$ and showed that it is isometric to the Feldman-Ilmanen-Knopf shrinking Ricci soliton \cite{FeldmanIlmanenKnopf} (see \cite{ConlonDeruelleSun-2019} for more results on gradient K$\ddot{a}$hler shrinking and expanding Ricci solitons).

One special feature observed by Munteanu-Wang \cite{MunteanuWang-2015.3} is that in four dimensional gradient Ricci soliton, the Riemann curvature $Rm$ can be bounded by $\Ric$ and $\na \Ric$:
\be\label{Rm controlled by Rc temp}
|Rm|\leq A_0\Big(|\Ric|+\frac{|\na \Ric|}{|\na f|}\Big),
\ee
for some universal positive constant $A_0$. Please see Lemma \ref{Rm to Rc} for the precise statement and proof of (\ref{Rm controlled by Rc temp}).

We will show that estimate (\ref{MW 2015 estimate}) is also true for four dimensional non-trivial gradient steady Ricci soliton with scalar curvature $S\to 0$ as $r \to \infty$. Under the additional assumption that $S$ has at most polynomial decay, Cao-Cui \cite{CaoCui-2014} proved that for any $a$ $\in (0,\frac{1}{2})$, there is a positive constant $C$ such that
$$|Rm|\leq CS^a.$$
Here is another main result of this paper:

\begin{thm}\label{Rm by S}Let $(M^4,g,f)$ be a four dimensional complete non Ricci flat gradient steady Ricci soliton with $\displaystyle\lim_{r\to \infty} S=0$. There exists a positive constant $c$ such that
\be\label{Rm by S ineq}
|Rm|\leq cS \text{  on  } M.
\ee
\end{thm}
%A Riemannian manifold $(M,g)$ is said to be Ricci pinched if for some $\varepsilon>0$, $\Ric\geq \varepsilon Sg$ on $M$. It is known that any $n$ dimensional complete Ricci pinched gradient steadier with non-negative sectional curvature must be flat if $n\geq 3$ (see \cite{ChowLuNi-2006}). Deng-Zhu \cite{DengZhu-2015} showed that any complete Ricci pinched gradient K$\ddot{a}$hler steadier with complex dimension $\geq 2$ is Ricci flat. Using theorem \ref{Rm by S} and the same argument in \cite{ChowLuNi-2006}, we have the following
%\begin{cor}
%Any four dimensional complete Ricci pinched gradient steady Ricci soliton is Ricci flat.
%\end{cor}

As an application, we show that the Riemann curvature $Rm$ of a four dimensional complete nontrivial gradient steady Ricci soliton decays exponentially if the potential $f$ is bounded from above and $\overline{\lim}_{r\to \infty} rS$ is sufficiently small.
\begin{thm} \label{decay estimate in dim 4} Let $(M^4,g,f)$ be a four dimensional complete non Ricci flat gradient steady Ricci soliton with $\displaystyle\lim_{r\to \infty} S=0$. Suppose the potential function $f$ is bounded from above by a constant and $\displaystyle\overline{\lim_{r\to\infty}} rS<\frac{1}{5A_0^2}$, where $A_0$ is the constant in (\ref{Rm controlled by Rc temp}). Then there is a constant $C>0$ such that
$$|Rm|(x)\leq Ce^{-r(x)} \text{ on  } M.$$
\end{thm}
%\begin{remark}
%Just like the case in theorem \ref{decay estimate in dim n}, the estimate is sharp and the above conclusion is still true under weaker condition on $f$, namely, $\overline{\lim}_{r\to\infty} r^{-1}f<1$.
%\end{remark}
The paper is organized as follows. In Section $2$, we introduce the basic preliminaries and notations needed in the subsequent sections. In Section $3$, we provide estimates for the growth of the potential function. Theorem \ref{decay estimate in dim n} will then be proved in Section $4$. In Section $5$, we will show Theorem \ref{Rm by S} and \ref{decay estimate in dim 4}.

{\sl Acknowledgement}: The author is greatly indebted to his advisor Prof. Jiaping Wang for his constant support, guidance and valuable discussions. The author would also like to thank Prof. Huai-Dong Cao, Prof. Ben Chow, Prof. Ovidiu Munteanu and Prof. Lei Ni for their comments and interests in this work. The author was partially supported by NSF grant DMS-1606820.

\section{preliminaries and notations}

Let $(M^n,g)$ be a connected smooth $n$ dimensional Riemannian manifold and $f$ be a smooth function on $M$. $(M^n,g,f)$ is said to be a gradient steady Ricci soliton with potential function $f$ if
\be\label{eq-RS-22}
\Ric+\na^2 f=0.
\ee
%Given any potential function $f$ which satisfies (\ref{eq-RS-22}), we consider the function $f_{\sigma}:=f+\sigma$, where $\sigma$ is any constant. It is clear that $f_{\sigma}$ also satisfies (\ref{eq-RS-22}). Our argument doesn't depend on which potential function we choose. Let us fix one and denote it by $f$ hereinafter.
$(M,g)$ is assumed to be complete as a Riemannian manifold throughout this paper. Ricci soliton is related to the self similar solution to the Ricci flow (solution which evolves by time dependent scaling and reparametrization of a fixed metric, see \cite{Chowetal-2007}). We consider the flow of the vector field $\na f$ and denote it by $\phi_t$, $\phi_0$ is the identity map. By a result of Zhang \cite{Zhang-2009}, $\phi_t$ exists for all time $t$ $\in \R$. Let $g(t):=\phi_t^*g$, then $g(t)$ is a solution to the Ricci flow with initial data $g(0)=g$.

 We use $S$ to denote the scalar curvature of the Riemannian manifold. For any smooth function $\omega$, the weighted Laplacian with respect to $\omega$ is defined to be the operator $\Delta_{\omega}:=\Delta-\na_{\na\omega}$. Fix a particular point $p_0$ $\in M$,  for any $x$ $\in M$, we denote the distance of $x$ from $p_0$ by $r =r(x) =d(x, p_0)$. The following equations are known for gradient steady Ricci soliton (see \cite{Chowetal-2007} or \cite{PetersenWylie-2010} for derivation of these formulas):
\be
S+\Delta f= 0,
\ee
\be\label{na S}
\na S=2\Ric(\na f),
\ee
\be\label{eqn of S}
\D S=-2|\Ric|^2 \text{ and }
\ee
\be\label{eqn of Rc}
\D R_{ij}=-2R_{iklj}R_{kl}.
\ee
It was proved by Hamilton \cite{Hamilton-1995} that $S+|\na f|^2=C$ for some constant $C$. When the metric is not Ricci flat, $S$ and hence $C$ are positive on $M$ (see the discussion after (\ref{eqn of nav})). We can normalize $C$ by scaling the metric and get:
%When the metric is not Ricci flat, we also have (by scaling the metric if necessary):
\be \label{eqn of naf}
S+|\na f|^2=1 \text{ and }
\ee
\be\label{eqn of f}
\Delta_f f=-1.
\ee
We define a function $v$ in the following way:
\be\label{lower bdd of v}
v:= -f.
\ee
From (\ref{eqn of naf}) and (\ref{eqn of f}), we have
\be\label{eqn of v}
\D v=1 \text{    and   }
\ee
\be \label{eqn of nav}
S+|\na v|^2=1.
\ee
Chen \cite{Chen-2009} showed that any complete ancient solution to the Ricci flow must have non-negative scalar curvature. As a result, the scalar curvature $S$ of a complete gradient steady Ricci soliton is non-negative (see also \cite{Zhang-2009}). Moreover by strong minimum principle, $S=0$ somewhere if and only if the steady soliton is Ricci flat. It is also known that any compact steady Ricci soliton must be Ricci flat (see \cite{Chowetal-2007}). Hence any complete non Ricci flat gradient steady Ricci soliton must be non-compact and have positive scalar curvature.

We give several well known examples of steady Ricci solitons:
\begin{example} Ricci flat manifolds or product of any two steady Ricci solitons.
\end{example}

\begin{example}\label{cigar soliton}\cite{Hamilton-1988}, \cite{Chowetal-2007} Hamilton's cigar soliton $\Sigma$: Hamilton constructed a two dimensional complete gradient steady soliton on $\R^2$. It is rotationally symmetric with positive sectional curvature. In the standard coordinate of $\R^2$, its metric can be written as follows (see \cite{Cao-2010})
$$g_{\Sigma}=\frac{4(dx^2+dy^2)}{1+x^2+y^2}.$$
The potential function $f$ is given by $f(x,y)=-\log(1+x^2+y^2)$. The scalar curvature $S$ decays like $e^{-r}$ as $r\to \infty$.
\end{example}

\begin{example} \cite{Bryant-2005}, \cite{Chowetal-2007} Bryant soliton: Bryant constructed an $n$ dimensional complete gradient steady soliton on $\R^n$, where $n\geq 3$. It is rotationally symmetric with positive sectional curvature. The scalar curvature $S$ decays like $r^{-1}$ as $r\to \infty$.
\end{example}
For more examples of steady Ricci solitons, see \cite{Cao-2010}. We will give a proof of (\ref{Rm controlled by Rc temp}). The following lemma is due to Munteanu-Wang \cite{MunteanuWang-2015.3} and stated in a slightly different form in \cite{MunteanuWang-2015.3} and \cite{CaoCui-2014}. For the sake of completeness, we sketch the proof here.
\begin{lma}\label{Rm to Rc}\cite{MunteanuWang-2015.3} Let $(M^4,g,f)$ be a four dimensional gradient Ricci soliton. There exists a universal positive constant $A_0$ such that if $\na f \neq 0$ at $q$ $\in M$, then at $q$
\be\label{Rm controlled by Rc}
|Rm|\leq A_0\Big(|\Ric|+\frac{|\na \Ric|}{|\na f|}\Big).
\ee
\end{lma}
\begin{proof}This proof was presented by Munteanu during one of his talks on \cite{MunteanuWang-2015.3}. Whenever $\na f \neq 0$, we consider an orthonormal frame $\displaystyle\{e_i\}_{i=1}^4$ with $e_4=\frac{\na f}{|\na f|}$. For any gradient Ricci soliton, by Ricci identity and (\ref{eq-RS-2})
$$R_{kl,i}-R_{ki,l}=R_{ilkj}f_j.$$
Hence for any $1\leq i,j,k\leq 4$,
$$|R_{ijk4}|\leq \frac{c|\na \Ric|}{|\na f|}.$$
Since the dimension is four, we have by direct computation,
$$R_{1221}=\frac{1}{2}(R_{11}+R_{22}-R_{33}-R_{44})+R_{3443}.$$
From the above identity, we see that Lemma \ref{Rm to Rc} holds.
\end{proof}

\section{Estimates on the potential function $f$}
We first provide some equivalent criteria for the properness of the potential function $f$ which will be used in later sections. Please see \cite{CarrilloNi-2009} and \cite{CaoChen-2012} for more results on linear growth of $f$.
\begin{lma}\label{eqv cond for proper of f}
Let $(M^n,g,f)$ be an n dimensional complete non Ricci flat steady gradient Ricci soliton with scalar curvature $S\to 0$ as $r\to \infty$. Then the following are equivalent:\\
\\
$(i)$ $\displaystyle\overline{\lim_{r\to \infty}} r^{-1}f<1$ ;\\
$(ii)$ $\displaystyle\lim_{r\to \infty} f=-\infty$ ;\\
$(iii)$ $\forall$ $\a$ $\in(0,1)$, there exists constant $D>0$ such that
\begin{equation*}
r+D\geq -f\geq \a r-D \text{  on  } M,
\end{equation*}
in particular $\displaystyle\lim_{r\to \infty} r^{-1}f=-1$.
\end{lma}
\begin{remark} By (\ref{eqn of naf}), we always have $-1\leq\displaystyle\overline{\lim_{r\to \infty}} r^{-1}f\leq 1$.
\end{remark}
\begin{proof}$(iii) \Rightarrow (ii)$ and $(ii) \Rightarrow (i)$ are obvious.\\
$(i) \Rightarrow (iii):$ The upper bound of $-f$ follows from (\ref{eqn of naf}). We now consider the lower bound. Since the scalar curvature $S$ decays at infinity, for any $\a$ $\in$ $(0,1)$ there exists a compact subset $K$ of $M$ such that
$$|\na f|\geq \a \text{  on  } M\setminus K.$$
We consider the flow of the vector field $\frac{\na f}{|\na f|^2}$ and denote it by $\psi_t$, $\psi_0$ is the identity map. Let $q$ $\in M\setminus K$ and for small non-negative $t$,
\be\label{distance v.s. t}
d(\psi_t(q),q)\leq \int_0^t\frac{1}{|\na f|(\psi_s(q))}ds\leq \frac{t}{\a}.
\ee
By short time existence of ordinary differential equation, $\psi_t(q)$ exists as long as it is in $M\setminus K$. Therefore we can define $T$ as follows
\be
T:=\sup\{a,   \psi_t(q)\in M\setminus K \text{  for all  } t \in [0,a]\}.
\ee
Since $K$ is compact, $T$ is positive. We claim that $T$ is finite. We first assume the claim and prove the lemma. Under $T<\infty$, we know that $\psi_T(q)$ exists and lies in $K$, and by (\ref{distance v.s. t})
\begin{eqnarray*}
-f(q)&=&T-f(\psi_T(q))\\
&\geq& \a(d(q,p_0)-d(\psi_T(q),p_0))-f(\psi_T(q))\\
&\geq& \a  d(q,p_0)-\a \sup_K d(p_0, \cdot)-\sup_K |f|.
\end{eqnarray*}
The above inequality also holds for $q$ $\in$ $K$. Thus we showed $(iii)$ is true. It remains to justify the claim, i.e. $T<\infty$. We argue by contradiction, suppose $T=\infty$. By (\ref{distance v.s. t}) and (\ref{eqn of naf}), for all $t$ $\in$ $[0, \infty)$,
\begin{eqnarray*}
d(\psi_t(q),q)&\leq& \int_0^t\frac{1}{|\na f|(\psi_s(q))}ds\\
&=& t + \int_0^t\frac{1-|\na f|}{|\na f|}ds\\
&=& t + \int_0^t\frac{S}{|\na f|(1+|\na f|)}ds\\
&\leq& t+ \frac{1}{\a}\int_0^t S(\psi_s(q)) ds.
\end{eqnarray*}
We have for all $t\geq0$
\be \label{dist upper bdd}
d(\psi_t(q),q)\leq t+ \frac{1}{\a}\int_0^t S(\psi_s(q)) ds.
\ee
We shall write $d(q,p_0)$ as $r(q)$, $r(\psi_t(q))$ as $r_t$, where $p_0$ is a fixed reference point. By triangular inequality and (\ref{eqn of naf}),
\begin{eqnarray}\notag
f(\psi_t(q))-f(q)&=&t \\
\label{t< rt}&\leq& d(\psi_t(q),q)\\
\notag &\leq& r_t+r_0.
\end{eqnarray}
Hence by (\ref{dist upper bdd})
\begin{eqnarray*}
r_t-r_0&\leq& d(\psi_t(q),q)\\
&\leq& f(\psi_t(q))-f(q)+ \frac{1}{\a}\int_0^t S(\psi_s(q)) ds\\
&\leq& f(\psi_t(q))-f(q)+ \frac{1}{\a}\int_0^{r_t+r_0} S(\psi_s(q)) ds.
\end{eqnarray*}
We also used (\ref{t< rt}) in the last inequality. Dividing the inequality by $r_t$, we deduce that
\be\label{for contradiction}
1\leq \frac{f(\psi_t(q))}{r_t}+\frac{r_0-f(q)}{r_t}+\frac{1}{\a r_t}\int_0^{r_t+r_0} S(\psi_s(q)) ds.
\ee
Since $f(\psi_t(q))-f(q)=t$, we see that $r_t\to \infty$ as $t\to \infty$. Using L Hospital's Rule and $\displaystyle\lim_{r\to \infty} S=0$,
\begin{eqnarray*}
\lim_{t\to \infty}\frac{1}{r_t}\int_0^{r_t+r_0} S(\psi_s(q)) ds&=&\lim_{x\to \infty}\frac{1}{x}\int_0^{x+r_0} S(\psi_s(q)) ds\\
&=&\lim_{x\to \infty} S(\psi_{x+r_0}(q))\\
&=&0.
\end{eqnarray*}
Now taking $\limsup$ $t\to \infty$ in (\ref{for contradiction}), we have $1\leq \displaystyle\overline{\lim_{t\to \infty}} r_t^{-1}f(\psi_t(q))$ which contradicts with $(i)$. We conclude that $T<\infty$ and finish the proof.
\end{proof}
We next show that $f+r$ is bounded from above and below provided that $f$ is proper and scalar curvature $S$ decays sufficiently fast.
\begin{lma}\label{f-r is bdd}
Let $(M^n,g,f)$ be an n dimensional non Ricci flat complete steady gradient Ricci soliton. Suppose that $\displaystyle\lim_{r\to \infty} f=-\infty$ and for some positive constant $C$
$$S\leq Ce^f \text{  on  } M.$$
Then there exists positive constant $c_0$ such that
$$r-c_0\leq -f\leq r+c_0 \text{  on  } M.$$
\end{lma}
\begin{proof}The upper bound of $-f$ again follows from (\ref{eqn of naf}). For the lower bound, there is a small $t_0$ such that on $\{f\leq t_0\}$,
\be\label{bdd of cef}
Ce^f\leq \frac{1}{2} \text{  and  }
\ee
\begin{eqnarray*}
|\na f|^2&=&1-S\\
&\geq& 1-Ce^f\\
&\geq& \frac{1}{2}.
\end{eqnarray*}
Let $z$ $\in$ $\{f<t_0\}$ and $t:=f(z)$. We again denote the flow of the vector field $\frac{\na f}{|\na f|^2}$ by $\psi_s$, $\psi_0$ is the identity map. By  short time existence of ordinary differential equation, $\psi_s(z)$ exists as long as it is in $\{f\leq t_0\}$.
The following are true:
$$f(\psi_s(z))-f(z)=s \text{  and  }$$
\begin{eqnarray*}
d(\psi_{t_0-t}(z),z)&\leq& \int_0^{t_0-t}|\dot{\psi_{\tau}}(z)| d\tau\\
&=&\int_0^{t_0-t}\frac{1}{|\na f|}d\tau\\
&=&\int_0^{t_0-t}\frac{1}{\sqrt{1-S}}d\tau\\
&\leq&\int_0^{t_0-t}\frac{1}{\sqrt{1-Ce^{\tau+t}}}d\tau\\
&=& t_0-t +\int_0^{t_0-t}\frac{1-\sqrt{1-Ce^{\tau+t}}}{\sqrt{1-Ce^{\tau+t}}}d\tau\\
&=& f(\psi_{t_0-t}(z))-f(z)\\
&& + \int_0^{t_0-t}\frac{Ce^{\tau+t}}{(1+\sqrt{1-Ce^{\tau+t}})\sqrt{1-Ce^{\tau+t}}}d\tau\\
&\leq& f(\psi_{t_0-t}(z))-f(z) +\int_0^{t_0-t} \sqrt{2}Ce^{\tau+t}d\tau\\
&\leq& f(\psi_{t_0-t}(z))-f(z) +\sqrt{2}Ce^{t_0}.
\end{eqnarray*}
From the above inequality and $f(\psi_{t_0-t}(z))=t_0$, we have
\begin{eqnarray*}
-f(z)&\geq& r(z)-r(\psi_{t_0-t}(z))-f(\psi_{t_0-t}(z))-\sqrt{2}Ce^{t_0}\\
     &\geq& r(z)-\sup_K r(\cdot)-\sup_K |f|-\sqrt{2}Ce^{t_0},
\end{eqnarray*}
where $K:=\{f\geq t_0\}$ is a compact set. Result follows.
\end{proof}

\section{Proof of theorem \ref{decay estimate in dim n}}
Motivated by the study of certain weighted elliptic equation and the choice of barrier functions by Deruelle \cite{Deruelle-2017} and Munteanu-Sung-Wang \cite{MunteanuSungWang-2017}, we estimate the decay rate of subsolution $u$ of $\D u\geq -Au^2$, where $A\geq 0$ is a constant.
\begin{lma}\label{prelim estimate for u}
Let $(M^n,g,f)$ be an $n$ dimensional non Ricci flat complete steady gradient Ricci soliton. Suppose that $\displaystyle\lim_{r\to \infty} f=-\infty$ and $u$ is a non-negative subsolution of the following differential inequality
\be\label{diff ineq for u}
\D u\geq -Au^2,
\ee
where A is a non-negative constant. If in addition $\displaystyle\overline{\lim_{r\to \infty}} r u=: B<\infty$ and $AB<1$, then there exists positive constant $C$ such that
\be\label{rough estimate for u}
u\leq C(v^2+1)e^{-v} \text{  on  } M,
\ee
where $v:=-f$.
\end{lma}
\begin{proof} Since $v=-f\to\infty$, $v>0$ near infinity. Pick any $\a$ $\in$ $(AB, 1)$, we will show that $v^{\a}u$ and $v^{3\a}e^{-v}$ are subharmonic and superharmonic functions respectively with respect to the operator $\Delta_{f+2\a\ln v}$, then we will apply the maximum principle to $v^{\a}u-bv^{3\a}e^{-v}$ for some well chosen constant $b$. For $v^{\a}u$ , we compute directly using (\ref{eqn of v}) and (\ref{diff ineq for u})
\begin{eqnarray*}
\D (v^{\a}u)&=& v^{\a}\D u + 2\langle \na v^{\a}, \na (v^{\a}uv^{-\a})\rangle+ u\D v^{\a}\\
&=&v^{\a}\D u +2\a\langle \na \ln v, \na (v^{\a}u)\rangle-2\a^2|\na v|^2v^{\a-2}u+ u\D v^{\a}\\
&\geq&-Av^{\a}u^2+2\a\langle \na \ln v, \na (v^{\a}u)\rangle-2\a^2|\na v|^2v^{\a-2}u\\
&&+\a v^{\a-1}u+\a(\a-1)|\na v|^2v^{\a-2}u\\
&=&\a v^{\a-1}u-Av^{\a}u^2-\a(\a+1)v^{\a-2}|\na v|^2u\\
&&+2\a\langle \na \ln v, \na (v^{\a}u)\rangle.
\end{eqnarray*}
Hence we get the following inequality
\begin{eqnarray*}
\Delta_{f+2\a\ln v} (v^{\a}u)&\geq& v^{\a-1}u\big[ \a-A uv-\a(\a+1)\frac{|\na v|^2}{v}\big]\\
&=&v^{\a-1}u\big[ \a-AB+A(B-uv)-\a(\a+1)\frac{|\na v|^2}{v}\big].
\end{eqnarray*}
By (\ref{eqn of naf}) and (\ref{eqn of nav}), $|\na f|=|\na v|\leq 1$, we see that $v\leq r+D$ for some constant $D$ and thus
$$\overline{\lim_{r\to \infty}} vu\leq \overline{\lim_{r\to \infty}} ru=B.$$

Using $\lim_{r\to \infty} v=\infty$, we have
\begin{eqnarray*}
\Delta_{f+2\a\ln v} (v^{\a}u)&\geq& v^{\a-1}u\frac{(\a-AB)}{4}\\
&\geq& 0
\end{eqnarray*}
outside some compact subset of $M$. For the function $v^{3\a}e^{-v}$, by (\ref{eqn of nav})
\begin{eqnarray}\notag
\D e^{-v}&=&-e^{-v}\D v+e^{-v}|\na v|^2\\
\label{eqn for e-v}
&=&e^{-v}(|\na v|^2-1)\\
\notag
&=&-Se^{-v}.
\end{eqnarray}
Direct computation also yields
\begin{eqnarray*}
\D v^{3\a}&=&3\a v^{3\a-1}\D v+3\a(3\a-1)v^{3\a-2}|\na v|^2\\
&=&3\a v^{3\a-1}+3\a(3\a-1)v^{3\a-2}|\na v|^2.
\end{eqnarray*}
Using (\ref{eqn of nav}) and $\displaystyle\lim_{r\to \infty} v=\infty$,
\begin{eqnarray*}
\D (v^{3\a}e^{-v})&=&v^{3\a}\D e^{-v} + e^{-v}\D v^{3\a}+2\langle \na v^{3\a}, \na e^{-v}\rangle\\
&=&-Sv^{3\a}e^{-v}+3\a(3\a-1)v^{3\a-2}e^{-v}|\na v|^2\\
&&+3\a v^{3\a-1}e^{-v}-6\a v^{3\a-1}e^{-v}|\na v|^2\\
&=&-Sv^{3\a}e^{-v}+3\a(3\a-1)v^{3\a-2}e^{-v}|\na v|^2\\
&&-3\a v^{3\a-1}e^{-v}+6\a v^{3\a-1}e^{-v}S\\
&=&-3\a v^{3\a-1}e^{-v}+3\a(3\a-1)v^{3\a-2}e^{-v}|\na v|^2\\
&&+Sv^{3\a}e^{-v}(-1+6\a v^{-1})\\
&=&v^{3\a-1}e^{-v}\Big[-3\a+3\a(3\a-1)\frac{|\na v|^2}{v}\Big]\\
&&+Sv^{3\a}e^{-v}(-1+6\a v^{-1})\\
&=&v^{3\a-1}e^{-v}(-3\a+ o(1))+Sv^{3\a}e^{-v}(-1+o(1))\\
&\leq&v^{3\a-1}e^{-v}(-3\a+ o(1)).
\end{eqnarray*}
\begin{eqnarray*}
-2\a\langle \na \ln v, \na (v^{3\a}e^{-v})\rangle &=&-6\a ^2v^{3\a-2}e^{-v}|\na v|^2 + 2\a v^{3\a-1}e^{-v}|\na v|^2\\
&=&-6\a ^2v^{3\a-2}e^{-v}|\na v|^2+ 2\a v^{3\a-1}e^{-v}\\
&&- 2\a v^{3\a-1}e^{-v}S\\
&\leq& 2\a v^{3\a-1}e^{-v}.
\end{eqnarray*}
We see that near infinity
\begin{eqnarray*}
\Delta_{f+2\a\ln v}(v^{3\a}e^{-v})&\leq&v^{3\a-1}e^{-v}(-\a+ o(1))\\
&<& 0.
\end{eqnarray*}
Hence $v^{3\a}e^{-v}$ is superharmonic with respect to $\Delta_{f+2\a\ln v}$. To proceed, we choose a large $R_0$ such that on $M\setminus B_{R_0}(p_0)$, $v>1$,
\be\label{diff ineq for vau}
\Delta_{f+2\a\ln v}(v^{\a}u)\geq 0 \text{  and  }
\ee
\be\label{diff ineq for v3ae-v}
\Delta_{f+2\a\ln v}(v^{3\a}e^{-v})<0.
\ee

Pick a large $b>0$ such that on $\partial B_{R_0}(p_0)$
\be
v^{\a}u-bv^{3\a}e^{-v}<0.
\ee
Let $Q:=v^{\a}u-bv^{3\a}e^{-v}$. Since $\a<1$ and $\displaystyle\overline{\lim_{r\to\infty}} vu\leq B$, we have $\displaystyle\lim_{r\to \infty}Q=0$. Moreover by (\ref{diff ineq for vau}) and (\ref{diff ineq for v3ae-v}), on $M\setminus B_{R_0}(p_0)$
$$\Delta_{f+2\a\ln v}Q>0.$$
By the maximum principle, we know that $Q\leq 0$ on $M\setminus B_{R_0}(p_0)$. Result follows.
\end{proof}
With the preliminary estimate in Lemma \ref{prelim estimate for u}, we are able to remove the quadratic factor $v^2$ in (\ref{rough estimate for u}) and thus improve the decay estimate of $u$.
\begin{prop}\label{sharp estimate for u}
Let $(M^n,g,f)$ be an $n$ dimensional non Ricci flat complete steady gradient Ricci soliton. Suppose that $\displaystyle\lim_{r\to \infty} f=-\infty$ and $u$ is a non-negative subsolution of the following differential inequality
\be\label{diff ineq for u ver 2}
\D u\geq -Au^2,
\ee
where A is a non-negative constant. If in addition $\displaystyle\overline{\lim_{r\to \infty}} r u=: B<\infty$ and $AB<1$, then u satisfies
$$u \leq Ce^{-v} \text{  on  }  M,$$
for some positive constant $C$, where $v:=-f$.
\end{prop}
\begin{remark}
By considering $v+\sigma$ for some large $\sigma$, one can show by similar arguments that both Lemma \ref{prelim estimate for u} and Proposition \ref{sharp estimate for u} are still true if the condition $\lim_{r\to\infty}f=-\infty$ is replaced by the boundedness of $f$ from above by a constant.
\end{remark}
\begin{remark} The condition $AB<1$ may look artificial but is critical indeed. In three dimensional Bryant steady soliton, $S^2\geq 2|\Ric|^2$ since the sectional curvature is non-negative (see \cite{MunteanuSungWang-2017}). Hence the scalar curvature $S\geq 0$ satisfies
$$\D S\geq -S^2.$$
However, $\displaystyle\lim_{r\to \infty} rS=\lim_{r\to \infty} -r^{-1}f=1$ (see \cite{Brendle-2013}). $AB=1$ and $S$ doesn't decay exponentially.
\end{remark}
For non-trivial gradient steady soliton with $\Ric\geq 0$ and $\lim_{r\to\infty} S=0$, we know that $S^2\geq|\Ric|^2$ and $\lim_{r\to\infty} f=-\infty$ (see \cite{CarrilloNi-2009}). We have an immediate consequence of Proposition \ref{sharp estimate for u} and Lemma \ref{f-r is bdd}:
\begin{cor} Let $(M^n,g,f)$ be an n dimensional complete non Ricci flat gradient steady Ricci soliton with $\Ric\geq 0$. Suppose that $\displaystyle\overline{\lim_{r\to\infty}} rS< \frac{1}{2}$, then there exists a constant $C>0$ such that
$$S\leq Ce^{-r} \text{  on  } M.$$
\end{cor}
\begin{proof}[Proof of Proposition \ref{sharp estimate for u}]
Firstly, we prove that $u e^{-v^{-1}}$ is subharmonic with respect to the operator $\Delta_{f-2v^{-1}}$ near infinity.
\begin{eqnarray*}
\D e^{-v^{-1}}&=&v^{-2}e^{-v^{-1}}\D v+ |\na v|^2\Big(-2v^{-3}e^{-v^{-1}}+ v^{-4}e^{-v^{-1}}\Big)\\
&=&v^{-2}e^{-v^{-1}}+ e^{-v^{-1}}|\na v|^2\Big(-2v^{-3}+ v^{-4}\Big).
\end{eqnarray*}
\begin{eqnarray*}
\D (u e^{-v^{-1}})&=& e^{-v^{-1}}\D u+2\langle \na u, \na e^{-v^{-1}}\rangle+ u\D e^{-v^{-1}}\\
&=& e^{-v^{-1}}\D u+ 2v^{-2}e^{-v^{-1}}\langle \na (e^{-v^{-1}}u e^{v^{-1}}), \na v\rangle\\
&&+ u\D e^{-v^{-1}} \\
&=& e^{-v^{-1}}\D u + 2v^{-2}\langle \na (e^{-v^{-1}}u ), \na v\rangle\\
&&-2v^{-4}u e^{-v^{-1}}|\na v|^2+ u\D e^{-v^{-1}}\\
&\geq& -Ae^{-v^{-1}}u^2-ue^{-v^{-1}}|\na v|^2\Big(2v^{-3}+ v^{-4}\Big)\\
&&+v^{-2}ue^{-v^{-1}}+ 2v^{-2}\langle \na (e^{-v^{-1}}u ), \na v\rangle.
\end{eqnarray*}
Hence
\be\label{diff ineq for ue-v-1}
\Delta_{f-2v^{-1}} (u e^{-v^{-1}})\geq v^{-2}ue^{-v^{-1}}\Big[1-Av^2u-|\na v|^2\big(2v^{-1}+ v^{-2}\big)\Big].
\ee
By Lemma \ref{prelim estimate for u}, $v^2u\to 0$ as $r\to \infty$. Therefore it is not difficult to see that $\Delta_{f-2v^{-1}} (u e^{-v^{-1}})\geq 0$ outside a compact subset of $M$.\\
\\
Secondly, we want to show that $e^{-v}e^{-4v^{-1}}$ is superharmonic with respect to the operator $\Delta_{f-2v^{-1}}$ near infinity. Using (\ref{eqn for e-v}) in Lemma \ref{prelim estimate for u}
$$\D e^{-v}=-Se^{-v}.$$

\begin{eqnarray*}
\D e^{-4v^{-1}}&=&4v^{-2}e^{-4v^{-1}}\D v +8|\na v|^2e^{-4v^{-1}}\big(-v^{-3}+2v^{-4}\big)\\
&=&4v^{-2}e^{-4v^{-1}}+8|\na v|^2e^{-4v^{-1}}\big(-v^{-3}+2v^{-4}\big).
\end{eqnarray*}
By $\lim_{r\to \infty} v=\infty$, outside a compact subset of $M$,
\begin{eqnarray*}
\D (e^{-v}e^{-4v^{-1}})&=& e^{-4v^{-1}}\D e^{-v}+ e^{-v}\D e^{-4v^{-1}}+2\langle \na e^{-v}, \na e^{-4v^{-1}}\rangle\\
&=& e^{-4v^{-1}}\D e^{-v}+ e^{-v}\D e^{-4v^{-1}}-8 e^{-v}e^{-4v^{-1}}v^{-2}|\na v|^2\\
&=& e^{-4v^{-1}}\D e^{-v}+ \Big[e^{-v}\D e^{-4v^{-1}}-8 e^{-v}e^{-4v^{-1}}v^{-2}\Big]\\
&&+8 e^{-v}e^{-4v^{-1}}v^{-2}S\\
&=& e^{-4v^{-1}}\D e^{-v}+8 e^{-v}e^{-4v^{-1}}v^{-2}S-4e^{-v}e^{-4v^{-1}}v^{-2}\\
&&+8|\na v|^2e^{-v}e^{-4v^{-1}}\big(-v^{-3}+2v^{-4}\big)\\
&=&-Se^{-v}e^{-4v^{-1}}+8 e^{-v}e^{-4v^{-1}}v^{-2}S- 4e^{-v}e^{-4v^{-1}}v^{-2}\\
&&+8|\na v|^2e^{-v}e^{-4v^{-1}}\big(-v^{-3}+2v^{-4}\big)\\
&=&e^{-v}e^{-4v^{-1}}v^{-2}\Big[- 4+8|\na v|^2\big(-v^{-1}+2v^{-2}\big)\Big]\\
&&+Se^{-v}e^{-4v^{-1}}(-1+8v^{-2})\\
&\leq&e^{-v}e^{-4v^{-1}}v^{-2}(-4+o(1)).
\end{eqnarray*}

\begin{eqnarray*}
-\langle \na (-2v^{-1}), \na (e^{-v}e^{-4v^{-1}})\rangle &=&-2v^{-2}\langle \na v, \na (e^{-v}e^{-4v^{-1}})\rangle \\
&=&2e^{-v}e^{-4v^{-1}}v^{-2}|\na v|^2\\
&&-8e^{-v}e^{-4v^{-1}}v^{-4}|\na v|^2\\
&\leq& 2e^{-v}e^{-4v^{-1}}v^{-2}.
\end{eqnarray*}

\begin{eqnarray*}
\Delta_{f-2v^{-1}}(e^{-v}e^{-4v^{-1}})&=&\D (e^{-v}e^{-4v^{-1}})-\langle \na (-2v^{-1}), \na (e^{-v}e^{-4v^{-1}})\rangle\\
&\leq&e^{-v}e^{-4v^{-1}}v^{-2}(-2+o(1))\\
&<&0.
\end{eqnarray*}
We showed that $e^{-v}e^{-4v^{-1}}$ is superharmonic with respect to $\Delta_{f-2v^{-1}}$ near infinity.\\
\\
To finish the proof, we choose a large $R_0$ such that on $M\setminus B_{R_0}(p_0)$, $v>1$,
\be\label{diff ineq for ue-v-1 ver2}
\Delta_{f-2v^{-1}} (u e^{-v^{-1}})\geq 0 \text{  and  }
\ee
\be\label{diff ineq for e-ve-4v-1 ver 2}
\Delta_{f-2v^{-1}}(e^{-v}e^{-4v^{-1}})<0.
\ee
Let $b>0$ be a large number such that on $\partial B_{R_0}(p_0)$
\be
u e^{-v^{-1}}-be^{-v}e^{-4v^{-1}}<0.
\ee
Let $Q:=u e^{-v^{-1}}-be^{-v}e^{-4v^{-1}}$. Obviously, we have $\displaystyle\lim_{r\to \infty}Q=0$. Moreover by (\ref{diff ineq for ue-v-1 ver2}) and (\ref{diff ineq for e-ve-4v-1 ver 2}), on $M\setminus B_{R_0}(p_0)$
$$\Delta_{f-2v^{-1}}Q>0.$$
By the maximum principle, $Q\leq 0$ on $M\setminus B_{R_0}(p_0)$ which implies that
$$u\leq be^{-v}e^{-3v^{-1}}\leq be^{-v}.$$
Choosing a larger $b$ if necessary, the above inequality holds globally on $M$.
\end{proof}

We are about to prove the main theorem. For the convenience of reader, we repeat the main theorem here:

\begin{thm*} Let $(M,g,f)$ be an $n$ dimensional non Ricci flat complete steady gradient Ricci soliton with $n\geq 2$. Suppose the potential function $f$ is bounded from above by a constant and $\displaystyle\overline{\lim_{r\to \infty}} r|Rm|<\frac{1}{5}$. Then there exists a positive constant $C$ such that
\be
|Rm|(x)\leq Ce^{-r(x)} \text{  on  } M.
\ee
\end{thm*}
\begin{proof}The potential function $f$ is bounded above by a constant and $|Rm|\to 0$ as $r\to \infty$, hence by Lemma \ref{eqv cond for proper of f}, $\displaystyle\lim_{r\to \infty} f=-\infty$ (one may also apply the inequality $S\geq Ce^{f}$ in \cite{MunteanuSungWang-2017}). From the evolution equation of $Rm$ of a solution to the Ricci flow ( \cite{ChowLuNi-2006} c.f. P.119 Lemma 2.51 ), we have in any orthonormal frame,

$$\D R_{abcd}=2(R_{apbq}R_{cpdq}-R_{apbq}R_{dpcq}+R_{apcq}R_{bpdq}-R_{apdq}R_{bpcq}).$$
By first Bianchi identity,
$$R_{cpdq}-R_{dpcq}=R_{cdpq} \text{  and  }$$
$$R_{apbq}=R_{aqbp}+R_{abpq}.$$
Therefore we have
\begin{eqnarray*}
R_{apbq}R_{cpdq}-R_{apbq}R_{dpcq}&=&R_{apbq}R_{cdpq}\\
&=&R_{aqbp}R_{cdpq}+R_{abpq}R_{cdpq}\\
&=&-R_{apbq}R_{cdpq}+R_{abpq}R_{cdpq}.
\end{eqnarray*}
We conclude that
\begin{eqnarray*}
2R_{apbq}R_{cpdq}-2R_{apbq}R_{dpcq}&=&2R_{apbq}R_{cdpq}\\
&=&R_{abpq}R_{cdpq}.
\end{eqnarray*}
The equation for $\D Rm$ becomes
$$\D R_{abcd}=R_{abpq}R_{cdpq}+2(R_{apcq}R_{bpdq}-R_{apdq}R_{bpcq}).$$
The above formula can also be derived using the evolution equation in \cite{Brendle-2010}. Using Kato's inequality and $2|Rm|\D|Rm|+2|\na|Rm||^2=2|\na Rm|^2+2\langle Rm,\D Rm\rangle$, we see that
\be\label{ineq for D Rm}
\D |Rm|\geq -5|Rm|^2
\ee
holds in the weak sense and on the set of points where $|Rm|$ is non-zero. The soliton is not Ricci flat and $S>0$, we may take $u=|Rm|>0$. By Proposition \ref{sharp estimate for u} and $\displaystyle\overline{\lim_{r\to \infty}} r|Rm|<5^{-1}$,
\begin{eqnarray*}
|Rm|&\leq& Ce^{-v}\\
&=&Ce^{f},\\
\end{eqnarray*}
which implies that $S\leq Ce^{f}$. We apply Lemma \ref{f-r is bdd} and conclude that $|Rm|\leq Ce^{-r+c_0}$ on $M$.

\end{proof}

\section{Proofs of theorem \ref{Rm by S} and \ref{decay estimate in dim 4}}

We recall the following estimate for the Ricci curvature of four dimensional gradient steady Ricci soliton by Cao-Cui \cite{CaoCui-2014}:
\begin{thm}\label{CaoCui estimate}\cite{CaoCui-2014} Let $(M^4,g,f)$ be a four dimensional complete non Ricci flat gradient steady Ricci soliton with $\displaystyle\lim_{r\to \infty} S=0$. Then $Rm$ is bounded and $\forall$ $a$ $\in (0,\frac{1}{2})$, there exists a positive constant $C$ such that
\be\label{CaoCui Rc by S}
|\Ric|\leq CS^{a} \text{ on } M.
\ee
\end{thm}
We first prove a slightly better estimate for the Ricci curvature.
\begin{lma}\label{Rc by S}Let $(M^4,g,f)$ be a four dimensional complete non Ricci flat gradient steady Ricci soliton with $\displaystyle\lim_{r\to \infty} S=0$. There exists a positive constant $A_1$ such that
\be\label{Rc by S ineq}
|\Ric|\leq A_1 S \text{  on  } M.
\ee
\end{lma}
\begin{remark}
The constant $A_1$ depends on $A_0$ in (\ref{Rm controlled by Rc}) and $\displaystyle\sup_{B_{R_0}(p_0)}\frac{|\Ric|}{S}$ for some large ball $B_{R_0}(p_0)$.
\end{remark}

\begin{proof} We claim that for some large $R_0>0$, on $M\setminus B_{R_0}(p_0)$
\be\label{ineq for Rc+RC2}
\D (|\Ric|+|\Ric|^2)\geq -(6A_0+20A_0^2)|\Ric|^2,
\ee
where $A_0$ is the universal constant in (\ref{Rm controlled by Rc}). We first assume the claim and prove the lemma. There exists a large constant $C>0$ such that the following inequalities are true
\be\label{condition 1 for C}
|\Ric|+|\Ric|^2-CS<0 \text{ on  } \partial B_{R_0}(p_0),
\ee
\be\label{condition 2 for C}
C\geq 3A_0+10A_0^2+1.
\ee
Using (\ref{ineq for Rc+RC2}), (\ref{condition 2 for C}) and (\ref{eqn of S}),
\begin{eqnarray*}
\D (|\Ric|+|\Ric|^2-CS)&=&\D (|\Ric|+|\Ric|^2)+ 2C|Ric|^2\\
&\geq&(2C-6A_0-20A_0^2)|\Ric|^2\\
&\geq& 2|\Ric|^2\\
&>&0.
\end{eqnarray*}
By Theorem \ref{CaoCui estimate}, we have $\displaystyle\lim_{r\to \infty} (|\Ric|+|\Ric|^2-CS)=0$. Using the maximum principle with boundary condition (\ref{condition 1 for C}), we see that on $M\setminus B_{R_0}(p_0)$, $|\Ric|+|\Ric|^2-CS\leq 0$. Choosing a bigger constant $C$ if necessary,
\begin{eqnarray*}
|\Ric|&\leq& |\Ric|+|\Ric|^2\\
&\leq& CS
\end{eqnarray*}
holds globally on $M$. To finish the argument, we need to show the inequality (\ref{ineq for Rc+RC2}) is true. By (\ref{eqn of Rc}) and (\ref{Rm controlled by Rc})
\begin{eqnarray*}
\D |\Ric|^2 &=& 2|\na\Ric|^2+2\langle \Ric, \D\Ric\rangle\\
&=&2|\na\Ric|^2-4R_{ij}R_{iklj}R_{kl}\\
&\geq& 2|\na\Ric|^2-4|Rm||\Ric|^2\\
&\geq& 2|\na\Ric|^2-4A_0|\Ric|^3-4A_0|\Ric|^2\frac{|\na \Ric|}{|\na f|}.
\end{eqnarray*}
Using Kato's inequality, we have
\be\label{ineq for abs val of Rc}
\D |\Ric|\geq -2A_0|\Ric|^2-2A_0|\Ric|\frac{|\na \Ric|}{|\na f|}.
\ee
We can further simplify the inequality for $|Ric|^2$,
\begin{eqnarray}
\notag
\D |\Ric|^2 &\geq& |\na\Ric|^2-4A_0|\Ric|^3+\Big(|\na\Ric|^2-4A_0|\Ric|^2\frac{|\na \Ric|}{|\na f|}\Big)\\
\notag
&=&|\na\Ric|^2-4A_0|\Ric|^3-\frac{4A_0^2|\Ric|^4}{|\na f|^2}\\
\notag
&&+\Big(|\na\Ric|^2-4A_0|\Ric|^2\frac{|\na \Ric|}{|\na f|}+\frac{4A_0^2|\Ric|^4}{|\na f|^2}\Big)\\
\notag
&=&|\na\Ric|^2-4A_0|\Ric|^3-\frac{4A_0^2|\Ric|^4}{|\na f|^2}\\
\notag
&&+\Big(|\na\Ric|-\frac{2A_0|\Ric|^2}{|\na f|}\Big)^2\\
\label{ineq for Rc2}
&\geq&|\na\Ric|^2-4A_0|\Ric|^3-\frac{4A_0^2|\Ric|^4}{|\na f|^2}.
\end{eqnarray}

Hence by (\ref{ineq for abs val of Rc})
\begin{eqnarray*}
\D (|\Ric|+|\Ric|^2)&\geq& -2A_0|\Ric|^2-2A_0|\Ric|\frac{|\na \Ric|}{|\na f|}\\
&&+|\na\Ric|^2-4A_0|\Ric|^3-\frac{4A_0^2|\Ric|^4}{|\na f|^2}\\
%&=&|\na\Ric|^2-2A_0|\Ric|\frac{|\na \Ric|}{|\na f|}-2A_0|\Ric|^2\\
%&&-4A_0|\Ric|^3-\frac{4A_0^2|\Ric|^4}{|\na f|^2}\\
&=&\Big(|\na\Ric|^2-2A_0|\Ric|\frac{|\na \Ric|}{|\na f|}+\frac{A_0^2|\Ric|^2}{|\na f|^2}\Big)\\
&&-\frac{A_0^2|\Ric|^2}{|\na f|^2}-2A_0|\Ric|^2-4A_0|\Ric|^3\\
&&-\frac{4A_0^2|\Ric|^4}{|\na f|^2}\\
&=&\Big(|\na\Ric|-\frac{A_0|\Ric|}{|\na f|}\Big)^2-\frac{A_0^2|\Ric|^2}{|\na f|^2}\\
&&-2A_0|\Ric|^2-4A_0|\Ric|^3-\frac{4A_0^2|\Ric|^4}{|\na f|^2}\\
&\geq&-\frac{A_0^2|\Ric|^2}{|\na f|^2}-2A_0|\Ric|^2-4A_0|\Ric|^3\\
&&-\frac{4A_0^2|\Ric|^4}{|\na f|^2}.
\end{eqnarray*}
By Theorem \ref{CaoCui estimate}, (\ref{eqn of naf}) and $\displaystyle\lim_{r\to \infty} S=0$, there exists a large positive $R_0$ such that on $M\setminus B_{R_0}(p_0)$
$$|\Ric|\leq 1 \text{  and }$$
$$|\na f|\geq \frac{1}{2}.$$
Using the above two inequalities, we have
\begin{eqnarray*}
\D (|\Ric|+|\Ric|^2)&\geq&-\frac{A_0^2|\Ric|^2}{|\na f|^2}-2A_0|\Ric|^2-4A_0|\Ric|^3\\
&&-\frac{4A_0^2|\Ric|^4}{|\na f|^2}\\
&\geq&-4A_0^2|\Ric|^2-2A_0|\Ric|^2-4A_0|\Ric|^2\\
&&-16A_0^2|\Ric|^2\\
&=&-(6A_0+20A_0^2)|\Ric|^2.
\end{eqnarray*}
We are done with the proof of the lemma.
\end{proof}
We can apply the Ricci curvature estimate to bound the curvature tensor $Rm$. To prepare for the proof, we start with the following computational lemma.
\begin{lma}\label{lma for ineq of S-2 Rc}Let $(M^4,g,f)$ be a four dimensional complete non Ricci flat gradient steady Ricci soliton with $\displaystyle\lim_{r\to \infty} S=0$. There exists a large constant $R_0>1$ such that for all $\lambda\geq0$, on $M\setminus B_{R_0}(p_0)$
\be\label{ineq for S-2 Rc}
\Delta_{f-2\ln S}(S^{-2}|\Ric|^2)\geq \frac{1}{32A_0^2}\Big(S^{-1}|Rm|+\lambda S^{-2}|\Ric|^2\Big)^2-c\lambda^2-c,
\ee
where $c$ is a positive constant depending only on $A_0$ and $A_1$ in (\ref{Rm controlled by Rc}) and (\ref{Rc by S ineq}) respectively.
\end{lma}
\begin{proof}
\begin{eqnarray*}
\D (S^{-2}|\Ric|^2)&=&S^{-2}\D |Ric|^2+2\langle \na S^{-2}, \na |\Ric|^2\rangle+ |\Ric|^2\D S^{-2}.
\end{eqnarray*}
By (\ref{eqn of S})
\begin{eqnarray*}
\D S^{-2}&=&-2S^{-3}\D S+6S^{-4}|\na S|^2\\
&=& 4S^{-3}|\Ric|^2+6S^{-2}|\na \ln S|^2.
\end{eqnarray*}

\begin{eqnarray*}
\langle \na S^{-2}, \na |\Ric|^2\rangle&=&-2S^{-3}\langle \na S, \na(S^{-2}|\Ric|^2S^2)\rangle\\
&=&-2\langle \na \ln S, \na(S^{-2}|\Ric|^2)\rangle-4|\na \ln S|^2S^{-2}|\Ric|^2.
\end{eqnarray*}
By Kato's inequality,
\begin{eqnarray*}
\langle \na S^{-2}, \na |\Ric|^2\rangle&=&-2S^{-3}\langle \na S, 2|\Ric|\na |\Ric|\rangle\\
&\geq&-4S^{-2}|\na \ln S||\Ric||\na \Ric|\\
&\geq&-\frac{1}{2}S^{-2}|\na \Ric|^2-8|\na \ln S|^2S^{-2}|\Ric|^2.
\end{eqnarray*}

\begin{eqnarray*}
\D (S^{-2}|\Ric|^2)&\geq& S^{-2}\D |Ric|^2-2\langle \na \ln S, \na(S^{-2}|\Ric|^2)\rangle\\
&&-12|\na \ln S|^2S^{-2}|\Ric|^2-\frac{1}{2}S^{-2}|\na \Ric|^2+ |\Ric|^2\D S^{-2}\\
&=&S^{-2}\D |Ric|^2-2\langle \na \ln S, \na(S^{-2}|\Ric|^2)\rangle\\
&&-12|\na \ln S|^2S^{-2}|\Ric|^2-\frac{1}{2}S^{-2}|\na \Ric|^2\\
&&+ 4S^{-3}|\Ric|^4+6|\na \ln S|^2S^{-2}|\Ric|^2\\
&\geq&S^{-2}\D |Ric|^2-2\langle \na \ln S, \na(S^{-2}|\Ric|^2)\rangle\\
&&-6|\na \ln S|^2S^{-2}|\Ric|^2-\frac{1}{2}S^{-2}|\na \Ric|^2.
\end{eqnarray*}
By Lemma \ref{Rc by S} and (\ref{na S}), $|\na S|=2|\Ric(\na f)|\leq 2|\Ric||\na f|\leq 2A_1 S$, where $A_1$ is the constant in (\ref{Rc by S ineq}). We now consider a large $R_0>0$ such that on $M\setminus B_{R_0}(p_0)$, $|\Ric|\leq 1$ and $|\na f|\geq \frac{1}{2}$.
Hence by (\ref{ineq for Rc2}) in Lemma \ref{Rc by S}, on $M\setminus B_{R_0}(p_0)$,
%\begin{eqnarray*}
%\Delta_{f-2\ln S}(S^{-2}|\Ric|^2)&\geq&\frac{1}{2}S^{-2}|\na \Ric|^2-4A_0S^{-2}|\Ric|^3\\
%&&-\frac{4A_0^2|\Ric|^4}{S^2|\na f|^2}-6|\na \ln S|^2S^{-2}|\Ric|^2\\
%&\geq&\frac{1}{2}S^{-2}|\na \Ric|^2-(4A_0+16A_0^2+6|\na \ln S|^2)S^{-2}|\Ric|^2\\
%&\geq&\frac{1}{2}S^{-2}|\na \Ric|^2-(4A_0+16A_0^2+24A_1^2)S^{-2}|\Ric|^2\\
%&\geq&\frac{1}{2}S^{-2}|\na \Ric|^2-(4A_0+16A_0^2+24A_1^2)A_1^2.
%\end{eqnarray*}
\begin{eqnarray*}
\Delta_{f-2\ln S}(\frac{|\Ric|^2}{S^2})&\geq&\frac{|\na \Ric|^2}{2S^2}-4A_0\frac{|\Ric|^3}{S^2}\\
&&-\frac{4A_0^2|\Ric|^4}{S^2|\na f|^2}-6|\na \ln S|^2\frac{|\Ric|^2}{S^2}\\
&\geq&\frac{|\na \Ric|^2}{2S^2}-(4A_0+16A_0^2+6|\na \ln S|^2)\frac{|\Ric|^2}{S^2}\\
&\geq&\frac{|\na \Ric|^2}{2S^2}-(4A_0+16A_0^2+24A_1^2)\frac{|\Ric|^2}{S^2}\\
&\geq&\frac{|\na \Ric|^2}{2S^2}-(4A_0+16A_0^2+24A_1^2)A_1^2.
\end{eqnarray*}
Using (\ref{Rm controlled by Rc}), $|Rm|^2\leq 2A_0^2(|\Ric|^2+\frac{|\na \Ric|^2}{|\na f|^2})\leq 2A_0^2(|\Ric|^2+4|\na \Ric|^2)$. We apply the inequality $2^{-1}(a+b)^2-b^2\leq a^2$ to conclude that for all $\lambda\geq 0$,
\begin{eqnarray*}
\Delta_{f-2\ln S}(S^{-2}|\Ric|^2)&\geq&\frac{1}{16A_0^2}S^{-2}|Rm|^2-\frac{1}{8}S^{-2}|\Ric|^2\\
&&-(4A_0+16A_0^2+24A_1^2)A_1^2\\
&\geq&\frac{1}{32A_0^2}\Big(S^{-1}|Rm|+\lambda S^{-2}|\Ric|^2\Big)^2-\frac{1}{8}S^{-2}|\Ric|^2\\
&&-\frac{\lambda^2}{16A_0^2}S^{-4}|\Ric|^4-(4A_0+16A_0^2+24A_1^2)A_1^2\\
&\geq&\frac{1}{32A_0^2}\Big(S^{-1}|Rm|+\lambda S^{-2}|\Ric|^2\Big)^2-\frac{\lambda^2}{16A_0^2}A_1^4\\
&&-\frac{1}{8}A_1^2-(4A_0+16A_0^2+24A_1^2)A_1^2\\
&\geq&\frac{1}{32A_0^2}\Big(S^{-1}|Rm|+\lambda S^{-2}|\Ric|^2\Big)^2-c\lambda^2-c,
\end{eqnarray*}
where $c$ only depends on $A_0$ and $A_1$. This ends the proof of the lemma.
%\notag
\end{proof}
Let us recall the statement of Theorem \ref{Rm by S}:
\begin{thm*} Let $(M^4,g,f)$ be a four dimensional complete non Ricci flat gradient steady Ricci soliton with $\displaystyle\lim_{r\to \infty} S=0$. There exists a positive constant $c$ such that
\be\label{second rm by S ineq}
|Rm|\leq cS \text{  on  } M.
\ee
\end{thm*}

\begin{remark}
The constant $c$ in (\ref{second rm by S ineq}) again depends on $A_0$ in (\ref{Rm controlled by Rc}), $A_1$ in (\ref{Rc by S ineq}) and $\displaystyle\sup_{B_{R_0}(p_0)}S^{-1}|Rm|$ for some large ball $B_{R_0}(p_0)$.
\end{remark}
\begin{remark}No assumption on the potential function $f$ is made in the above theorem. In particular, $f$ is not known to be proper in this setting. To construct a cut off function for the maximum principle argument, we use the distance function $r$ together with an appropriate Laplacian comparison theorem in \cite{WeiWylie-2009}.
\end{remark}
\begin{proof}The argument is essentially due to Munteanu-Wang \cite{MunteanuWang-2015.3} and Cao-Cui \cite{CaoCui-2014}. We know that by (\ref{ineq for D Rm})
$$\D |Rm|\geq -5|Rm|^2.$$
%\be\label{ineq for Rm}
%\D |Rm|\geq -5|Rm|^2
%\ee
\begin{eqnarray*}
\D (S^{-1}|Rm|) &=&S^{-1}\D |Rm|+2\langle \na S^{-1}, \na (SS^{-1}|Rm|)\ra\\
&&+|Rm|\D S^{-1}\\
&=&S^{-1}\D |Rm|-2\langle \na \ln S, \na (S^{-1}|Rm|)\rangle\\
&&-2|\na \ln S|^2S^{-1}|Rm|+|Rm|\D S^{-1}\\
&\geq&-5S^{-1}|Rm|^2-2\langle \na \ln S, \na (S^{-1}|Rm|)\rangle\\
&&+|Rm|(2S^{-2}|\Ric|^2+2S^{-1}|\na \ln S|^2)\\
&&-2|\na \ln S|^2S^{-1}|Rm|\\
&\geq&-5S^{-1}|Rm|^2-2\langle \na \ln S, \na (S^{-1}|Rm|)\rangle.
\end{eqnarray*}
Hence
\be
\Delta_{f-2\ln S}(S^{-1}|Rm|)\geq -5S^{-1}|Rm|^2.
\ee
We apply Lemma \ref{lma for ineq of S-2 Rc} and take $\lambda=192 A_0^2$, on $M\setminus B_{R_0}(p_0)$
%\begin{eqnarray}
%\notag
%\Delta_{f-2\ln S}\Big(S^{-1}|Rm|+\lambda S^{-2}|\Ric|^2\Big)&\geq&\big(\frac{\lambda}{32A_0^2}-5S\big)\Big(S^{-1}|Rm|+\lambda S^{-2}|\Ric|^2\Big)^2\\
%\notag
%&&-c\lambda^3-c\lambda\\
%\notag
%&=&\big(6-5S\big)\Big(S^{-1}|Rm|+\lambda S^{-2}|\Ric|^2\Big)^2\\
%\notag
%&&-c\lambda^3-c\lambda\\
%\label{ineq for S-1 Rm + s-2 Rc}
%&\geq&\Big(S^{-1}|Rm|+\lambda S^{-2}|\Ric|^2\Big)^2-c,
%\end{eqnarray}
\begin{eqnarray}
\notag
\Delta_{f-2\ln S}\Big(\frac{|Rm|}{S}+\lambda \frac{|\Ric|^2}{S^2}\Big)&\geq&\big(\frac{\lambda}{32A_0^2}-5S\big)\Big(\frac{|Rm|}{S}+\lambda \frac{|\Ric|^2}{S^2}\Big)^2\\
\notag
&&-c\lambda^3-c\lambda\\
\notag
&=&\big(6-5S\big)\Big(\frac{|Rm|}{S}+\lambda \frac{|\Ric|^2}{S^2}\Big)^2\\
\notag
&&-c\lambda^3-c\lambda\\
\label{ineq for S-1 Rm + s-2 Rc}
&\geq&\Big(\frac{|Rm|}{S}+\lambda \frac{|\Ric|^2}{S^2}\Big)^2-c,
\end{eqnarray}
where $c$ only depends on $A_0$ and $A_1$. Let $W:=\Big(S^{-1}|Rm|+\lambda S^{-2}|\Ric|^2\Big)$ and $G:=\phi^2 W$, where $\phi$ is a non-negative cut off function. By (\ref{ineq for S-1 Rm + s-2 Rc}), on the set where $\phi$ is positive,
\begin{eqnarray*}
\Delta_{f-2\ln S}G&=&\phi^2\Delta_{f-2\ln S} W+ 4\phi\la \na \phi,\na (\phi^2W\phi^{-2})\ra+ W\Delta_{f-2\ln S}\phi^2\\
&=&\phi^2\Delta_{f-2\ln S} W+4\la \na \ln \phi, \na G\ra-8|\na\phi|^2W\\
&&+ W\Delta_{f-2\ln S}\phi^2\\
&=&\phi^2\Delta_{f-2\ln S} W+4\la \na \ln \phi, \na G\ra\\
&&+(2\phi\D\phi+4\phi\la \na \ln S,\na \phi\ra-6|\na\phi|^2)W\\
&\geq&\phi^2W^2-c\phi^2+4\la \na \ln \phi, \na G\ra\\
&&+(2\phi\D\phi-4\phi^2|\na \ln S|^2-7|\na\phi|^2)W.
\end{eqnarray*}
By Lemma \ref{Rc by S} and (\ref{na S}), $|\na S|=|2\Ric(\na f)|\leq 2A_1 S$, we deduce that
\begin{eqnarray}
\label{ineq for G}
\phi^2\Delta_{f-2\ln S}G&\geq& G^2-c\phi^4+4\phi\la\na\phi, \na G\ra\\
\notag &&+(2\phi\D\phi-16\phi^2A_1^2-7|\na\phi|^2)G.
\end{eqnarray}
Let $R\geq R_0\geq 1$ and $\psi :[0,\infty)\rightarrow \R$ be a smooth real valued function satisfying the following:
$0\leq\psi\leq 1$, $\psi'\leq 0$,
\[ \psi(t)=\begin{cases}
      1 & 0\leq t\leq 1 \\
      0 & 2\leq t\\
   \end{cases}
\]
and
$$|\psi''(t)|+|\psi'(t)|\leq c \text{  for all  } t\geq 0.$$
We take $\phi(x):=\psi(\frac{r(x)}{R})$, then
\be
|\na \phi|=\frac{|\psi'|}{R}\leq \frac{c}{R}.
\ee
Using the Laplacian comparison theorem for smooth metric measure space with $\Ric_f\geq 0$ and $|\na f|\leq 1$ (\cite{WeiWylie-2009} c.f. Theorem 1.1 a)), we have
$$\D r\leq \frac{3}{r}+1.$$

\begin{eqnarray*}
\D \phi &=& \frac{\psi'}{R}\D r+\frac{\psi''}{R^2}|\na r|^2\\
&\geq&-\frac{c}{R}\big(\frac{3}{R}+1\big)-\frac{c}{R^2}\\
&\geq&-\frac{c_1}{R}.
\end{eqnarray*}
\begin{eqnarray}
\label{ineq for D phi2}
2\phi\D \phi-16\phi^2A_1^2-7|\na\phi|^2&\geq&-\frac{c_2}{R}-16A_1^2\\
\notag&\geq& -c,
\end{eqnarray}
where $c$ is a constant depending on $A_1$. Suppose $G$ attains its maximum at $q$. If $q$ $\in$ $\overline{B_{R_0}(p_0)}$, then
$$G\leq G(q)\leq W(q)\leq \sup_{\overline{B_{R_0}(p_0)}}W,$$
right hand side is obviously independent of $R$. If $q$ $\in$ $M\setminus \overline{B_{R_0}(p_0)}$, then by (\ref{ineq for G}) and (\ref{ineq for D phi2}), we have
\begin{eqnarray*}
0&\geq& \phi^2(q)\Delta_{f-2\ln S}G(q)\\
&\geq&G^2(q)-c-cG(q).
\end{eqnarray*}
From this we see that there exists a positive constant $C$ depending only on $A_0$ and $A_1$ such that
$$G\leq G(q)\leq C.$$
Result then follows by letting $R\to \infty$.
\end{proof}

%The constant $c$ in Theorem \ref{Rm by S} may depend on the local geometry of $(M,g)$ and provide no information on the size of $r|Rm|$ near infinity if one simply assumes $\overline{\lim}_{r\to \infty}rS$ is small. To have a better control on $\overline{\lim}_{r\to \infty}r|Rm|$, we look for a universal positive constant in front of the scalar curvature $S$. We need to refine our previous estimate on the Ricci curvature:

The constant $c$ in Theorem \ref{Rm by S} may depend on the local geometry of $(M,g)$ and provide no information on the size of $r|Rm|$ near infinity if one simply assumes $\overline{\lim}_{r\to \infty}rS$ is small. To prove Theorem \ref{decay estimate in dim 4} by Theorem \ref{decay estimate in dim n} and \ref{Rm by S}, we look for a universal positive constant in front of the scalar curvature $S$ and then estimate the size of $\overline{\lim}_{r\to \infty}r|Rm|$. To have a better control on $\overline{\lim}_{r\to \infty}r|Rm|$, we need to refine our previous estimate on the Ricci curvature:

\begin{prop}\label{Rc by A0S}Let $(M^4,g,f)$ be a four dimensional complete non Ricci flat gradient steady Ricci soliton with $\displaystyle\lim_{r\to \infty} S=0$. Suppose that $\displaystyle\lim_{r\to \infty}f=-\infty$ and $\displaystyle\overline{\lim_{r\to\infty}} rS<\infty$, then there exists a positive constant $C$ such that
\be
|\Ric|\leq A_0S+Cv^{-\frac{3}{2}}
\ee
outside some compact subset of $M$, where $v:=-f$. In particular,
\be
\overline{\lim_{r\to\infty}}r|\Ric|\leq A_0\overline{\lim_{r\to\infty}} rS.
\ee
\end{prop}
\begin{proof}Since $\displaystyle\overline{\lim_{r\to\infty}} rS<\infty$, we have by Lemma \ref{eqv cond for proper of f} and Theorem \ref{Rm by S} there exists a constant $C_0$ such that outside a compact set,
\be\label{C0 estimate of Rm}
|\Ric|\leq c|Rm|\leq CS\leq \frac{C_0}{v}.
\ee
By viewing steady soliton as a solution to the Ricci flow and Shi's derivative estimate \cite{ChowLuNi-2006},
\be\label{C1 estimate of Rm}
|\na\Ric|\leq c|\na Rm|\leq \frac{C_1}{v^{\frac{3}{2}}}.
\ee
Let $R_0$ be a large positive number such that on $M\setminus B_{R_0}(p_0)$, $|\na f|\geq \frac{1}{2}$, $v>1$, (\ref{C0 estimate of Rm}) and (\ref{C1 estimate of Rm}) hold, moreover, $v$ satisfies
$$\frac{1}{2}-\frac{15}{4}v^{-1}\geq 0.$$
We then choose a big positive constant $C$ such that $C\geq 4A_0C_0C_1+1$ and
$$|\Ric|-A_0S-Cv^{-\frac{3}{2}}<0 \text{  on  } \partial B_{R_0}(p_0).$$
By (\ref{eqn of S}) in Section $2$, (\ref{ineq for abs val of Rc}) in Lemma \ref{Rc by S}, (\ref{C0 estimate of Rm}) and (\ref{C1 estimate of Rm}), we have
\begin{eqnarray*}
\D(|\Ric|-A_0S-Cv^{-\frac{3}{2}})&\geq& -2A_0|\Ric|^2-2A_0|\Ric|\frac{|\na \Ric|}{|\na f|}+2A_0|\Ric|^2\\
&&+\frac{3C}{2}v^{-\frac{5}{2}}\D v-\frac{15C}{4}|\na v|^2v^{-\frac{7}{2}}\\
&\geq&-2A_0|\Ric|\frac{|\na \Ric|}{|\na f|}+\frac{3C}{2}v^{-\frac{5}{2}}-\frac{15C}{4}v^{-\frac{7}{2}}\\
&\geq&-4A_0C_0C_1v^{-\frac{5}{2}}+\frac{3C}{2}v^{-\frac{5}{2}}-\frac{15C}{4}v^{-\frac{7}{2}}\\
&=&v^{-\frac{5}{2}}\Big(\frac{3C}{2}-4A_0C_0C_1-\frac{15C}{4}v^{-1}\Big)\\
&\geq&v^{-\frac{5}{2}}\Big(C-4A_0C_0C_1\Big)\\
&\geq&v^{-\frac{5}{2}}\\
&>&0.
\end{eqnarray*}
Since $\displaystyle\lim_{r\to \infty}(|\Ric|-A_0S-Cv^{-\frac{3}{2}})=0$, by the maximum principle,
$$|\Ric|-A_0S-Cv^{-\frac{3}{2}}\leq 0 \text{  on  }      M\setminus B_{R_0}(p_0).$$
\end{proof}
For the reader's convenience, we recall the statement of Theorem \ref{decay estimate in dim 4}:
\begin{thm*}Let $(M^4,g,f)$ be a four dimensional complete non Ricci flat gradient steady Ricci soliton with $\displaystyle\lim_{r\to \infty} S=0$. Suppose the potential function $f$ is bounded from above by a constant and $\displaystyle\overline{\lim_{r\to\infty}} rS<\frac{1}{5A_0^2}.$ Then there is a constant $C>0$ such that
$$|Rm|\leq Ce^{-r} \text{ on  } M.$$
\end{thm*}
\begin{proof}Since $S\to 0$ as $r\to \infty$ and $f$ is bounded from above, we know by Lemma \ref{eqv cond for proper of f} that $f\to -\infty$ as $r\to \infty$. By Proposition \ref{Rc by A0S}, Shi's estimate (\ref{C1 estimate of Rm}) and Lemma \ref{Rm to Rc},
$$|Rm|\leq A_0^2S+CA_0v^{-\frac{3}{2}}+2C_1A_0v^{-\frac{3}{2}}.$$
It implies that
\begin{eqnarray*}
\displaystyle\overline{\lim_{r\to \infty}} r|Rm|&\leq& A_0^2\overline{\lim_{r\to\infty}} rS\\
&<& \frac{1}{5}.
\end{eqnarray*}
Hence by Theorem \ref{decay estimate in dim n}, $|Rm|$ decays exponentially.

\end{proof}

\end{document}